\documentclass[12pt]{article}
\usepackage{amsthm, amssymb, amsmath, mathrsfs, fullpage}

\numberwithin{equation}{section}
\newtheorem{theorem}[equation]{Theorem}
\newtheorem{lemma}[equation]{Lemma}
\newtheorem{prop}[equation]{Proposition}
\newtheorem{cor}[equation]{Corollary}

\def\QQ{\mathbb{Q}}
\def\RR{\mathbb{R}}
\def\ZZ{\mathbb{Z}}

\title{A construction of polynomials with squarefree discriminants}
\author{Kiran S. Kedlaya}
\date{May 4, 2011}

\begin{document}

\maketitle

\begin{abstract}
For any integer $n \geq 2$ and any nonnegative integers $r,s$ with $r+2s = n$,
we give an unconditional construction of infinitely many monic irreducible polynomials
of degree $n$ with integer coefficients having squarefree discriminant and exactly $r$ real roots.
These give rise to number fields of degree $n$, signature $(r,s)$, Galois group $S_n$,
and squarefree discriminant;
we may also force the discriminant to be coprime to any given integer.
The number of fields produced with discriminant in the range $[-N, N]$ is at least $c N^{1/(n-1)}$.
A corollary is that for each $n \geq 3$, infinitely many quadratic number fields
admit everywhere unramified degree $n$ extensions whose normal closures have Galois group $A_n$. 
This generalizes results of Yamamura, who treats the case $n = 5$, and
Uchida and Yamamoto, who allow general $n$ but do not control the real place.
\end{abstract}

\section{Introduction and results}

Throughout this paper, fix an integer $n \geq 2$, and let $P(x)$ denote a monic polynomial of degree $n$
with integer coefficients. The \emph{discriminant} $\Delta(P)$ is the integer with the property that if
we factor $P(x) = (x - r_1) \cdots (x - r_n)$ over some algebraically closed field, then
\[
\Delta(P) = \prod_{1 \leq i < j \leq n} (r_i - r_j)^2.
\]
Since $\Delta(P)$ is symmetric in the roots of $P$, it can be expressed as a polynomial
in the coefficients of $P$; this polynomial has integer coefficients and turns out to be irreducible.

When $P$ is irreducible and the integer $\Delta(P)$ is squarefree, the number field
$K = \mathbb{Q}[x]/(P(x))$ has ring of integers $\ZZ[x]/(P(x))$,
the Galois closure $L$ of $K$
has Galois group $S_n$, and $L$ is everywhere unramified over its subfield $\QQ(\sqrt{\Delta(P)})$ (see
\cite{egm}, or \cite[Theorem~1]{kondo} for a slightly stronger statement).
These properties suggest the question of how often $\Delta(P)$ is squarefree. When the coefficients of $P$
are chosen suitably randomly, this is expected to occur with probability $\prod_p a_p$ where $p$ runs
over prime numbers
and $a_p$ denotes the probability that $\Delta(P)$ is not divisible by $p^2$. These probabilities have been computed by
Brakenhoff \cite{ash}:
\[
a_p = 
\begin{cases}
\frac{1}{2} & p = 2, n \geq 2 \\
1 - \frac{1}{p^2} & p > 2, n=2 \\
1 - \frac{2}{p^2} + \frac{1}{p^3} & p > 2, n=3 \\
1 - \frac{1}{p} + \frac{(p-1)^2 (1 - (-p)^{-n+2})}{p^2(p+1)}
& p > 2, n \geq 4.
\end{cases}
\]
Unfortunately, while it is easy to prove by sieving arguments that a randomly chosen integer is squarefree with the expected
probability of $6/\pi^2$ (see for instance \cite[18.6]{bib:hw}), 
it seems quite difficult to prove that a polynomial takes squarefree values with
the expected probability unless the degree is very small compared to the number of variables.
For example, for univariate polynomials, this is known in all degrees under  
the $abc$ conjecture by a theorem of Granville \cite{granville},
but unconditionally only up to degree 3 by work
of Hooley \cite{bib:hooley}.
Granville's conditional theorem was extended to multivariate polynomials by Poonen
\cite{poonen}; that result implies that under $abc$, $\Delta(P)$ is squarefree with the expected probability.
This remains true, with a suitably adjusted probability, if one imposes local conditions at finitely many places (including the infinite place).

However, without assuming any conjectures, it is not trivial to establish even the existence of infinitely many
polynomials of a given degree with squarefree discriminant. This is given by the following theorem.
\begin{theorem} \label{T:squarefree}
Let $n \geq 2$ be an integer, let $r,s$ be nonnegative integers with $r + 2s = n$, and let $S$ be a finite
set of primes. Then there exist infinitely many monic irreducible polynomials $P(x)$ of degree $n$
with integer coefficients having exactly $r$ real roots, such that $\Delta(P)$ is squarefree and not divisible
by any of the primes in $S$. More precisely, for some $c>0$ (depending on $n$ and $S$), the number of distinct squarefree discriminants produced
in the range $[-N, N]$ is at least $cN^{1/(n-1)}$.
\end{theorem}
The key idea in the proof is to construct $P$ in a special way to make squarefree sieving on the discriminant easier.
For $a_1,\dots,a_{n-1},b \in \QQ$, put
\begin{equation}
Q_a(x) = n(x-a_1/n)(x - a_2) \cdots (x-a_{n-1}), \quad
P_{a,b}(x) = b + \int_0^x Q_a(t)\,dt, \quad
\Delta_{a,b} = \Delta(P_{a,b}).
\end{equation}
Then $\Delta_{a,b}$ is 
(up to sign) the product of the evaluations
of $Q_a$ at the roots of $P_{a,b}$, which is 
(up to sign) the resultant of $P_{a,b}$
and $Q_a$. However, one may also compute the resultant by evaluating $P_{a,b}$
at each root of $Q_a$:
\[
\Delta_{a,b} = \pm (n^n P_{a,0}(a_1/n) + n^n b) \prod_{i=2}^{n-1} (P_{a,0}(a_i) + b).
\]
We then choose $a_1,\dots,a_{n-1}$ so that for $b$ of a certain form, the apparent probability that
$P_{a,b}$ has the desired properties is positive. Since $\Delta_{a,b}$ factors as a product of linear polynomials in $b$,
it is tractable to carry out the sieving argument to confirm that the apparent probability is correct;
however, in lieu of doing the sieving by hand, we appeal to a very general squarefree sieve set up by
Helfgott \cite{helfgott}.
(The argument is a bit simpler if one does not 
insist on the number of real roots; see Proposition~\ref{P:squarefree}.)

From Theorem~\ref{T:squarefree}, we obtain at once the following corollary.
\begin{cor} \label{C:squarefree1}
With notation as in Theorem~\ref{T:squarefree},
there exist infinitely many number fields $K$ of degree $n = r+2s$ and signature $(r,s)$
such that the Galois closure $L$ of $K$ has Galois group $S_n$ over $\mathbb{Q}$, and the discriminant
of $K$ is squarefree and not divisible by any of the primes in $S$.
More precisely, for some $c>0$ (depending on $n$ and $S$), the number of distinct such discriminants produced
in the range $[-N, N]$ is at least $cN^{1/(n-1)}$.
\end{cor}
As noted earlier, these conditions ensure that 
$L$ is an unramified $A_n$-extension of $\QQ(\sqrt{\Delta(K)})$.
We thus obtain the following.
\begin{cor} \label{C:squarefree2}
Let $S$ be a finite set of primes.
For each $n \geq 3$, there exist
infinitely many real quadratic fields unramified above the primes in $S$
and admitting an $A_n$-extension unramified at all finite and infinite primes.
More precisely, for some $c>0$ (depending on $n$ and $S$), the number of distinct such discriminants produced
in the range $[-N, N]$ is at least $cN^{1/(n-1)}$.
\end{cor}

Corollary~\ref{C:squarefree2} generalizes work of several authors.
Uchida \cite{bib:uchida}
and Yamamoto \cite{bib:yamamoto}
proved independently that there exist infinitely
many real and infinitely many imaginary quadratic fields with
$A_n$-extensions unramified at all finite primes.
They did so by constructing number fields of the form $\QQ[x]/(x^n+ax+b)$
and showing that under suitable conditions on $a$ and $b$, their
normal closures are frequently $S_n$-extensions
of $\QQ$ unramified over a quadratic subfield. However, these number
fields fail to be totally real for $n \geq 4$, so the construction
does not produce the result of Corollary~\ref{C:squarefree2}. The case $n=5$ of
Corollary~\ref{C:squarefree2} was
obtained by Yamamura \cite{bib:yamamura}; however, Yamamura's
construction produces number fields of the form $\QQ[x]/(P(x))$ where
the polynomial $P$ is of a very special form which does not generalize
to higher degree. (Note that none of the aforementioned constructions guarantee squarefree discriminants.)

One may also deduce from Corollary~\ref{C:squarefree1} that the number of $S_n$-number fields
of signature $(r,s)$ with discriminant in $[-N, N]$ is bounded below by $c N^{1/(n-1)}$.
This is far inferior to what can be shown using other methods:
for $n \leq 5$ the number of such fields is known to be $O(N)$ by work of Bhargava \cite{bhargava4, bhargava5},
while for larger $n$, Ellenberg and Venkatesh \cite{ellenberg-venkatesh} have given a lower bound of $c N^{1/2 + 1/n^2}$.
However, the fields we produce have squarefree discriminants and 
monogenic rings of integers, which is not guaranteed by these
other constructions (although Bhargava's method should allow for squarefree sieving; see
\cite{bhargava-new}).

\section{Construction of squarefree discriminants}

In this section, we give the proof of a weaker form of Theorem~\ref{T:squarefree},
in which we do not control the number of real roots. This helps isolate the essential features of the
construction.
\begin{prop} \label{P:squarefree}
Let $n \geq 2$ be an integer and let $S$ be a finite
set of primes. Then there exist infinitely many monic irreducible polynomials $P(x)$ of degree $n$
with integer coefficients such that $\Delta(P)$ is squarefree and not divisible
by any of the primes in $S$. More precisely, for some $c>0$ (depending on $n$ and $S$), the number of distinct discriminants produced
in the range $[-N, N]$ is at least $cN^{1/(n-1)}$.
\end{prop}
The approach to the proof is to consider polynomials $P_{a,b}$ as in the introduction
with $a_1,\dots,a_{n-1} \in \ZZ$ chosen to satisfy the following conditions.
\begin{enumerate}
\item[(i)]
The polynomial $P_{a,0}$ has integral coefficients.
\item[(ii)]
For each $p \in S$, there exists $b_p \in \ZZ$ such that $\Delta_{a,b_p}$ is not divisible by $p$.
\item[(iii)]
For each $p \notin S$, there exists $b \in \ZZ$ such that $\Delta_{a,b}$ is not divisible by $p^2$.
(This ensures that $\Delta_{a,b}$ is squarefree with positive probability.)
\item[(iv)]
There exist $p_1 \notin S$ and $b_1 \in \ZZ$ such that $P_{a,b_1}$ is irreducible modulo $p_1$; in particular,
$\Delta_{a,b_1}$ is not divisible by $p_1$.
(This ensures that $P_{a,b_1}$ is irreducible for $b \equiv b_1 \pmod{p_1}$.)
\item[(v)]
The polynomials $n^n P_{a,0}(a_1/n) + n^n b, P_{a,0}(a_2) + b, \dots, P_{a,0}(a_{n-1}) + b$ in $b$
are pairwise coprime. (This ensures that $\Delta_{a,b}$ is squarefree as a polynomial in $b$.)
\end{enumerate}
Given such a choice of $a_1,\dots,a_{n-1}$,
let $T$ be the set of all $b \in \ZZ$ such that $b \equiv b_p \pmod{p}$ for each $p \in S$ and $b \equiv b_1 \pmod{p_1}$.
For $b \in T$, $P_{a,b}$ is irreducible with integer coefficients,
$\Delta_{a,b}$ is not divisible by any prime in $S$,
and by a sieving argument (see Lemma~\ref{L:squarefree sieve1} below), 
$\Delta_{a,b}$ is squarefree with positive probability.
Since no value of $\Delta_{a,b}$ occurs for more than $n-1$ choices of $b$,
by taking $b$ of size $O(N^{1/(n-1)})$, we obtain at least $c N^{1/(n-1)}$ squarefree values of $\Delta_{a,b}$ in the range $[-N, N]$ for some fixed $c>0$.

The sieving argument required in this argument is essentially no harder than
proving the density of squarefree integers; however, we have nothing to add to Helfgott's general presentation of the
squarefree sieve, so we defer to it instead.
\begin{lemma} \label{L:squarefree sieve1}
Let $m$ be a positive integer. Suppose that $c_1,d_1,\dots,c_k,d_k \in \ZZ$ are such that
the polynomial $A(x) = (c_1 x + d_1) \cdots (c_k x + d_k)$ is squarefree. For $p$ prime, put
\[
a(p) = \frac{\#\{x \in \{0,\dots,p^2-1\}: A(x) \not\equiv 0 \pmod{p^2}\}}{p^2}
\]
Suppose that $a(p) > 0$ for all $p$. Then 
\[
\lim_{N \to \infty} \frac{\#\{x \in \{-N, \dots, N\}: A(x) \mbox{ is squarefree}\}}{2N+1}
= \prod_p a(p) > 0.
\]
\end{lemma}
Note that $\prod_p a(p)$ converges to a positive limit  because
$1 - k/p^2 \leq a(p) \leq 1$ for all $p$ not dividing $\prod_{i<j} (c_i d_j - c_j d_i)$. 
\begin{proof}
See \cite[Proposition~3.4]{helfgott}.
\end{proof}

It remains to check that conditions (i)-(v) can be enforced via certain congruence conditions on
$a_1,\dots,a_{n-1}$. To make these congruences compatible with the requirement
that the derivative of $P_{a,b}$ have only rational roots,
we need some auxiliary calculations.

\begin{lemma} \label{L:irreducible}
Let $p_0$ be a prime not dividing $n(n-1)$. Then there exist infinitely many primes $p_1$ modulo which
the polynomial $R(x) = x^n - p_0^{n-1}x + p_0$ is irreducible and its derivative
$R'(x) = nx^{n-1} - p_0^{n-1}$ splits into distinct linear factors.
\end{lemma}
\begin{proof}
The polynomial $R'$ has splitting field $L = \QQ(\zeta_{n-1}, n^{1/(n-1)})$,
in which $p_0$ does not ramify because $p_0$ does not divide $n(n-1)$.
Thus $R$ is an Eisenstein polynomial
with respect to any prime above $p_0$ in $L$;
in particular, $R$ is irreducible over $L$.
By the Chebotarev density theorem, there exist infinitely many prime ideals of $L$ of absolute
degree 1 modulo which $R$ is irreducible; the norm of any such prime ideal is a prime number of the desired form.
\end{proof}
\begin{lemma} \label{L:distinct values}
For any field $F$ of characteristic zero, there exist $a_1,\dots,a_{n-1} \in F$ such that
the values $P_{a,0}(a_1/n), P_{a,0}(a_2),\dots,P_{a,0}(a_{n-1})$ are pairwise distinct.
\end{lemma}
\begin{proof}
For $a_1,\dots,a_{n-1} \in \QQ$ with $a_1 < \cdots < a_{n-1}$, 
\[
P_{a,0}(a_{n-1}) - P_{a,0}(a_{n-2}) = \int_{a{n-1}}^{a_{n-2}} (t - a_1/n)(t-a_2) \cdots (t-a_{n-1})\,dt \neq 0
\]
because the integrand has constant sign over the interval. Consequently, the polynomial $P_{a,0}(a_{n-1}) - P_{a,0}(a_{n-2})$ in $a_1,\dots,a_{n-1}$ is nonzero. By similar reasoning, each of the factors of the polynomial
\[
\prod_{i=2}^{n-1} (P_{a,0}(a_i) - P_{a,0}(a_1/n))
\prod_{2 \leq i < j \leq n-1} (P_{a,0}(a_j) - P_{a,0}(a_i))
\]
in $a_1,\dots,a_{n-1}$ is not zero, so the product is not the zero polynomial either. Since $F$ is infinite, the claim follows.
\end{proof}

We are now ready to complete the proof of Proposition~\ref{P:squarefree}.
\begin{proof}[Proof of Proposition~\ref{P:squarefree}]
By the previous discussion, it is sufficient to exhibit $a_1,\dots,a_{n-1} \in \ZZ$ for which
conditions (i)-(v) hold. 
Choose $p_0, R, p_1$ as in Lemma~\ref{L:irreducible}, making sure that $p_1 \notin S \cup \{1,\dots,n\}$.
Apply Lemma~\ref{L:distinct values} to choose $a'_1,\dots,a'_{n-1} \in \QQ$
for which 
\begin{equation} \label{eq:values}
P_{a',0}(a'_1/n), P_{a',0}(a'_2),\dots,P_{a',0}(a'_{n-1})
\end{equation}
are pairwise distinct. We can then choose a prime $p_2 \notin S \cup \{1,\dots,n,p_1\}$ for which the values
in \eqref{eq:values} have well-defined and distinct reductions modulo $p_2$. We now choose $a_1,\dots,a_{n-1}$ according to the following
conditions (which is evidently possible).
\begin{enumerate}
\item[(a)] The integer $a_1$ is coprime to $n$, and is divisible by $n-1$ and by all primes of $S$ not dividing $n$.
The integers $a_2, \dots, a_{n-1}$ are divisible by $n!$
and by all primes in $S$.
\item[(b)] The quantities $a_1/n, a_2, \dots, a_{n-1}$ are congruent modulo $p_1$ to the roots of $R'(x)$
(in some order).
\item[(c)] We have $a_i \equiv a'_i \pmod{p_2}$ for $i=1,\dots,n-1$.
\end{enumerate}
We then have
\[
Q_a(x) \equiv (nx-a_1) x^{n-2} = nx^{n-1} - a_1 x^{n-2} \pmod{n!},
\]
so the coefficient of $x^{m-1}$ in $Q_a(x)$ is divisible by $m$ for $m=1, \dots, n$.
Consequently, $P_{a,0}$ has integer coefficients. This verifies (i).
Suppose next that $p$ is a prime which either belongs to $S$ or is at most $n$. For $b \in \ZZ$,
$P_{a,0}(a_i) + b \equiv b \pmod{p}$ for $i=2,\dots,n-1$, while
\[
n^n P_{a,0}(a_1/n) + n^n b \equiv \begin{cases} a_1^n \pmod{p} & \mbox{if $p$ divides $n$} \\
n^n b \pmod{p} & \mbox{otherwise.} \end{cases}
\]
Consequently, $\Delta_{a,1}$ is not divisible by $p$; this verifies (ii), as well as (iii) for primes
$p \leq n$. For $p > n$ not in $S$, there are at least $p-n+1$ choices of $b \in \{0,\dots,p-1\}$
for which $\Delta_{a,b}$ is not even divisible by $p$, since each linear factor of $\Delta_{a,b}$
rules out exactly one choice of $b$. This verifies (iii). Since (b) implies (iv) and (c) implies (v), the
needed conditions are enforced, so the proof described above goes through.
\end{proof}

\section{Controlling real roots}

Note that the proof of Proposition~\ref{P:squarefree} cannot be used to deduce Theorem~\ref{T:squarefree}
because it does not allow any control of the number of real roots of the polynomial $P_{a,b}$.
Indeed, the function $P_{a,0}(x)$ is monotonic for $x$ large and for $x$ small,
so $P_{a,b}$ has at most 2 real roots for $|b|$ sufficiently large.
To obtain Theorem~\ref{T:squarefree}, we must modify the proof of Proposition~\ref{P:squarefree} so that
$b$ can be chosen to be a rational number within a suitable interval.

\begin{proof}[Proof of Theorem~\ref{T:squarefree}]
Define $p_0,R,p_1,p_2,a'_1,\dots,a'_{n-1}$ as in the proof of Proposition~\ref{P:squarefree}.
Let $U$ be the set of $A = (A_1,\dots,A_{n-1}) \in \RR^{n-1}$ for which there exists $B \in \RR$ such that
the polynomial $P_{A,B}$ has exactly $r$ real roots; the set $U$ is easily seen to be open and nonempty.
It is also homogeneous: if $(A_1,\dots,A_{n-1})$ belongs to $U$, then so does
$(\lambda A_1, \dots, \lambda A_{n-1})$ for any $\lambda \in \RR$. 
In other words, $U$ is the inverse image in $\RR^{n-1}$ of a nonempty open subset of the $(n-2)$-dimensional
projective space over $\RR$. In that space, the images of the tuples consisting of integers
satisfying conditions (a)-(c) of the proof of Proposition~\ref{P:squarefree} form a dense subset;
we may thus choose $(A_1,\dots,A_{n-1}) \in U$ satisfying these conditions.
Fix such a choice hereafter; there then exists a nonempty interval $I$ such that
$P_{A,B}$ has exactly $r$ real roots for all $B \in I$.

We wish to take $a_i = Aq$ for some positive integer $q \equiv 1 \pmod{n! p_1 p_2 \prod_{p \in S} p}$.
These still satisfy conditions (a)-(c) of the proof of Proposition~\ref{P:squarefree},
so conditions (i)-(v) are enforced. We may then sieve again (see Lemma~\ref{L:squarefree sieve2})
to show that among $b$ of this form congruent to $b_p$ modulo $p$ for each $p \in S$ and congruent to $b_1$ modulo $p_1$,
those for which $\Delta_{a,b}$ is squarefree occur with positive probability. 
Since the map $b \to \Delta_{a,b}$ is at most $(n-1)$-to-one, if we consider all $q \leq N^{1/(n-1)}$, 
the number of squarefree discriminants obtained is at least
$1/(n-1)$ times the maximum number of squarefree values of $\Delta_{a,b}$
achieved for any \emph{single} choice of $q$ in this range. This yields a lower bound of $c N^{1/(n-1)}$
discriminants of absolute value at most $N$, proving Theorem~\ref{T:squarefree}.
\end{proof}

In this case, the sieving argument needed is an easy (because we only consider polynomials composed of
linear factors) variant of the squarefree sieving for homogeneous binary forms introduced by Greaves \cite{greaves}.
However, we again prefer to apply Helfgott's machine rather than get into the details.

\begin{lemma} \label{L:squarefree sieve2}
Let $k,n,t$ be positive integers, and let $I$ be a nonempty open interval.
Suppose that $c_1,d_1,\dots,c_k,d_k \in \ZZ$ are such that
the polynomial $A(x,y) = (c_1 x + d_1 y) \cdots (c_k x + d_k y)$ is squarefree. 
For $p$ prime, if $p$ divides $t$, then put
\[
a(p) = 
\frac{\#\{x \in \{0,\dots,p^2-1\}: A(x,1) \not\equiv 0 \pmod{p^2}\}}{p^2};
\]
otherwise, put
\[
a(p) = 
\frac{\#\{(x,y) \in \{0,\dots,p^2-1\}^2: \gcd(x,y) = 1, A(x,y^n) \not\equiv 0 \pmod{p^2}\}}{p^2(p-1)^2}.
\]
Suppose that $a(p) > 0$ for all $p$.
Define
\[
S_N = \{(x,y) \in \ZZ^2: 1 \leq y \leq N, y \equiv 1 \pmod{t}, \gcd(x,y) = 1,
x/y^n \in I\}.
\]
Then 
\[
\lim_{N \to \infty} \frac{\#\{(x,y) \in S_N: A(x,y^n) \mbox{ is squarefree}\}}{\#S_N}
= \prod_p a_p > 0.
\]
\end{lemma}
\begin{proof}
This follows from \cite[Corollary~3.3]{helfgott} modulo matching up notation, which we now explain
(following the model of \cite[Proposition~3.5]{helfgott}). We
define a \emph{soil} in the sense of \cite[\S 3.2]{helfgott} by taking
\begin{gather*}
\mathscr{P} = \{ \mbox{primes in $\ZZ$}\}, \qquad
\mathscr{A} =  S_N, \\
r(x,y) = \{p \in \mathscr{P}: p^2 | A(x,y^n)\}, \qquad
f(x,y,d) = \begin{cases} 1 & d = \emptyset \\ 0 & \mbox{otherwise.} \end{cases}
\end{gather*}
Put $X = \#S_N$; note that $X$ is asymptotic to a constant times $N^{n+1}$ as $N \to \infty$.
For $d \subset \mathscr{P}$ finite, write $h(d)$ for $\prod_{p \in d} p^2$;
this function evidently satisfies conditions (h1) and (h2) of \cite[\S 3.2]{helfgott}.
Let $K(d)$ be the number of $(x,y) \in \mathscr{A}$ with $d \subseteq r(x,y)$;
since there are only finitely many primes that can divide more than one factor of
$A(x,y^n)$, 
we have $K(d) \leq C_0 X/h(d)$ for some $C_0 > 0$ (dependent on everything but $N$ and $d$).
This means that condition
(A1) of \cite[\S 3.2]{helfgott} holds with $C_1 = 1, C_2 = 0$.

For $d$ with $h(d) \leq N^{1/4}$, put $g(d,\emptyset) = h(d) \prod_{p \in d} a(p)$;
we may then write
\[
K(d) = X \frac{g(d,\emptyset)}{h(d)} + r_{d,\emptyset}
\]
with $|r_{d,\emptyset}| \leq N h(d) \leq N^{5/4}$. (Namely, for each $y$, we are counting values of $x$
in an interval satisfying certain congruence conditions modulo $h(d)$; the difference between
this count and the expected value is at most 1 in each congruence class.)
By \cite[Corollary~3.3]{helfgott},
the difference
\[
 \prod_p a_p - \frac{\#\{(x,y) \in S_N: A(x,y^n) \mbox{ is squarefree}\}}{X}
\]
is bounded in absolute value by
\[
C_1 \sum_{d: h(d) > N^{1/4}} \frac{3^{\#d}}{h(d)} + 
X^{-1} \sum_{d: h(d) \leq N^{1/4}} |r_{d, \emptyset}| + X^{-1} C_2 \sum_{p > N} \frac{X}{p^2}
\]
for some $C_1, C_2 > 0$ independent of $N$.
The first and third terms evidently tend to 0 as $N \to \infty$; the second does also because
it is bounded above by $X^{-1} N^{1/4} N^{5/4} \leq C_3 N^{-n+1/2}$ for some $C_3>0$ independent of $N$.
This proves the claim.
\end{proof}

\section{Further remarks}

We have not attempted to improve upon the lower bound of $cN^{1/(n-1)}$
in Theorem~\ref{T:squarefree} or its corollaries. Nakagawa \cite{bib:nak1}, \cite{bib:nak2} attempted to show that 
the number of real quadratic fields admitting 
$A_n$-extensions unramified at all finite primes with discriminants
in $[-N, N]$ is at least $c_nN^{(n+1)/(2n-2)}$
for all $n$, but these proofs were later retracted. 
This order of growth would be obtained if one could show that the discriminant of an arbitrary (not necessarily monic)
integer polynomial of degree $n$ is squarefree with positive probability \emph{and} that the number of distinct discriminants obtained is at least a
fixed 
positive fraction of the number of polynomials considered. 
Even if one assumes the \emph{abc} conjecture, so that the first issue is resolved by Poonen's theorem,
the second issue remains: for any construction of polynomials involving more than one free parameter,
it is nontrivial to ensure that the same discriminant does not occur too many times.

It would also be of interest to extend our results by allowing
restrictions on the splitting of some finite places. Our method is unsuitable for this purpose: a splitting requirement
would constitute an additional restriction on the reduction of a
polynomial modulo specific primes, which is hard to integrate with the
requirement that the derivative of the polynomial have rational roots
modulo any prime. Indeed, for primes less than $n$, certain splitting 
requirements on the polynomial alone (e.g., that it factors completely) are
incompatible with having a squarefree discriminant, because there are not enough
residual roots available for them to be pairwise distinct.

\subsection*{Acknowledgments}
This paper is a revised version of our unpublished 2003 preprint ``Number fields
with squarefree discriminant and prescribed signature,'' written while the author
was supported by an NSF Postdoctoral Fellowship.
The revisions were made during the spring 2010 MSRI program in Arithmetic Statistics,
during which the author was supported by NSF CAREER grant DMS-0545904,
DARPA grant HR0011-09-1-0048, MIT (NEC Fund,
Cecil and Ida Green Career Development Professorship), and UC San Diego
(Stefan E. Warschawski Professorship).
Thanks to Manjul Bhargava, Chandan Singh Dalawat, Noam Elkies, and Harald Helfgott for helpful discussions, and especially
to Bhargava for encouraging us to revisit the 2003 manuscript.


\begin{thebibliography}{99}

\bibitem{ash}
A. Ash, J. Brakenhoff, and T. Zarrabi, Equality of polynomial and field discriminants, \textit{Exper.
Math.} \textbf{16} (2007), 367--374.

\bibitem{bhargava4}
M. Bhargava, The density of discriminants of quartic rings and fields, 
\textit{Ann. of Math.} \textbf{162} (2005), 1031--1063.

\bibitem{bhargava5}
M. Bhargava, The density of discriminants of quintic rings and fields, 
\textit{Ann. of Math.} \textbf{172} (2010), 1559--1591.

\bibitem{bhargava-new}
M. Bhargava, The geometric squarefree sieve and unramified nonabelian
extensions of quadratic fields, preprint (2011).

\bibitem{ellenberg-venkatesh}
J.S. Ellenberg and A. Venkatesh,
The number of extensions of a number field with fixed degree and bounded discriminant,
\textit{Ann. of Math.} \textbf{163} (2006), 723--741. 

\bibitem{egm}
J. Elstrodt, F. Grunewald, and J. Mennicke, 
On unramified $A_m$-extensions of quadratic number fields,
\textit{Glasgow Math. J.} \textbf{27} (1985), 31--37.

\bibitem{granville}
A. Granville, $ABC$ allows us to count squarefrees,
\textit{Intl. Math. Res. Notices} \textbf{1998}, 991--1009.

\bibitem{greaves}
G. Greaves, Power-free values of binary forms, 
\textit{Quart. J. Math. Oxford} \textbf{43} (1992), 45--65.

\bibitem{bib:hw}
G.H. Hardy and E.M. Wright, An Introduction to the Theory of 
Numbers, fourth edition, Oxford University Press (London), 1965.

\bibitem{helfgott}
H.A. Helfgott, On the square-free sieve,
\textit{Acta Arith.} \textbf{115} (2004), 349--402. 

\bibitem{bib:hooley}
C. Hooley, On the power free values of polynomials, \textit{Mathematika}
\textbf{14} (1967), 21--26.

\bibitem{kondo}
T. Kondo, Algebraic number fields with the discriminant equal to that
of a quadratic number field, \textit{J. Math. Soc. Japan}
\textbf{47} (1995), 31--36.

\bibitem{bib:nak1}
J. Nakagawa, Binary forms and orders of algebraic number fields,
\textit{Invent. Math.} \textbf{97} (1989), 219--235; erratum,
\textit{ibid.} \textbf{105} (1991), 443.

\bibitem{bib:nak2}
J. Nakagawa, Binary forms and unramified $A_n$-extensions of quadratic
fields, \textit{J. reine angew. Math.} \textbf{406} (1990), 167--178;
correction, \textit{ibid.} \textbf{413} (1991), 220.

\bibitem{poonen}
B. Poonen, Squarefree values of multivariate polynomials,
\textit{Duke Math. J.} \textbf{118} (2003), 353--373.

\bibitem{bib:uchida}
K. Uchida, Unramified extensions of quadratic number fields, II,
\textit{T\^ohoku Math. J.} \textbf{22} (1970), 220--224.


\bibitem{bib:yamamoto}
Y. Yamamoto, On unramified Galois extensions of quadratic number fields,
\textit{Osaka J. Math.} \textbf{7} (1970), 57--76.

\bibitem{bib:yamamura}
K. Yamamura, On unramified Galois extensions of real quadratic number
fields, \textit{Osaka J. Math.} \textbf{23} (1986), 471--478.


\end{thebibliography}
\end{document}